\documentclass[12pt]{amsproc}

\theoremstyle{plain}
\newtheorem{theorem}{Theorem}[section]
\newtheorem{corollary}[theorem]{Corollary}
\newtheorem{lemma}[theorem]{Lemma}
\newtheorem{proposition}[theorem]{Proposition}

\theoremstyle{definition}

\theoremstyle{remark}

\newcommand{\U }{\,\mathcal {U}}
\newcommand{\W }{\,\mathcal {W}}
\newcommand{\h}{\tilde v}

\author[E. Detomi]{Eloisa Detomi}
\address{Dipartimento di Matematica ``Tullio Levi-Civita", Universit\`a di Padova \\
 Via Trieste 63\\ 35121 Padova \\ Italy}
\email{detomi@math.unipd.it}
\author[M. Morigi]{Marta Morigi}
\address{Dipartimento di Matematica, Universit\`a di Bologna\\
Piazza di Porta San Donato 5 \\ 40126 Bologna\\ Italy}
\email{marta.morigi@unibo.it}
\author[P. Shumyatsky]{Pavel Shumyatsky}
\address{Department of Mathematics, University of Brasilia\\
Brasilia-DF \\ 70910-900 Brazil}
\email{pavel@unb.br}

\thanks{This research was supported by FAPDF and CNPq-Brazil}


\begin{document}

\title[Words of Engel type II]{Words of Engel type are concise in residually finite groups. Part II}

\begin{abstract} This work is a natural follow-up of the article \cite{dms-conciseness}. Given a group-word $w$ and a group $G$, the verbal subgroup $w(G)$ is the one generated by all $w$-values in $G$. The word $w$ is called concise if $w(G)$ is finite whenever the set of $w$-values in $G$ is finite. It is an open question whether every word is concise in residually finite groups. Let $w=w(x_1,\ldots,x_k)$ be a multilinear commutator word, $n$ a positive integer and $q$ a prime power. In the present article we show that  the word $[w^q,{}_ny]$ is concise in residually finite groups (Theorem \ref{bb}) while the word $[w,{}_ny]$ is boundedly concise in residually finite groups (Theorem \ref{aa}).
\end{abstract}

\maketitle
\section{Introduction}
Let $w=w(x_1,\dots,x_k)$ be a group-word. Given a group $G$, we denote by $w(G)$ the verbal subgroup corresponding to the word $w$, that is, the subgroup generated by the set $G_w$ of all values $w(g_1,\ldots,g_k)$, where $g_1,\ldots,g_k$ are elements of $G$. The word $w$ is called concise if $w(G)$ is finite whenever the set of $w$-values in $G$ is finite. More generally, a word $w$ is called concise in a class of groups $\mathcal X$ if
$w(G)$ is finite whenever the set of $w$-values in $G$ is finite for a group $G\in\mathcal X$. In the sixties Hall raised the problem whether all  words are concise. In 1989 S. Ivanov \cite{ivanov} (see also \cite[p.\ 439]{ols}) solved the problem in the negative. On the other hand, the problem for residually finite groups remains open (cf. Segal \cite[p.\ 15]{Segal} or Jaikin-Zapirain \cite{jaikin}). In recent years a number of new results with respect to this problem were obtained (see \cite{AS1, gushu,  fealcobershu, dms-conciseness, dms-bounded}). The work \cite{dms-conciseness}  deals with conciseness of words of Engel type in residually finite groups. 
Recall that  multilinear  commutator words, also known under the name of outer commutator words, are precisely the words that can be written in the form of multilinear Lie monomials. The multilinear commutator words were shown to be concise (in the class of all groups) by J.\,C.\,R. Wilson \cite{jwilson}. Set $[x,{}_1y]=[x,y]=
x^{-1}y^{-1}xy$ and $[x,{}_{i+1}y]=[[x,{}_iy],y]$ for $i\geq1$. The word $[x,{}_ny]$ is called the $n$th Engel word. One of the results obtained in 
 \cite{dms-conciseness}  says that for any multilinear commutator word $w=w(x_1,\dots,x_k)$ and any positive integer $n$ the word $[w,{}_ny]$ is concise in residually finite groups. 

We say that a word $w$ is boundedly concise in a class of groups $\mathcal X$ if for every integer $m$ there exists a number $\nu=\nu(\mathcal X,w,m)$ such that whenever $|G_w|\leq m$ for a group $G\in\mathcal X$ it always follows that $|w(G)|\leq\nu$. Fernandez-Alcober and Morigi showed that every word which is concise in the class of all groups is actually boundedly concise \cite{fernandez-morigi}. It was conjectured in \cite{fealcobershu} that every word which is concise in the class of residually finite groups is boundedly concise.

One purpose of the present article is to show that the words $[w,{}_ny]$, where $w=w(x_1,\dots,x_k)$ is a multilinear commutator word, are boundedly concise in residually finite groups.

\begin{theorem}\label{aa} Suppose that $w=w(x_1,\ldots,x_k)$ is a multilinear commutator word and $n$ is a positive integer. The word $[w,{}_ny]$ is boundedly concise in residually finite groups.
\end{theorem}

The other purpose of this paper is to establish conciseness of somewhat more complicated words. 

\begin{theorem}\label{bb} Suppose that $w=w(x_1,\ldots,x_k)$ is a multilinear commutator word. For any prime-power $q$ and any positive integer $n\geq1$ the word $[w^q,{}_ny]$ is concise in residually finite groups.
\end{theorem}

We do not know if the word $[w^q,{}_ny]$ in the above theorem is boundedly concise in residually finite groups. Note that at present, among all words that are known to be concise in residually finite groups, the only ones whose bounded conciseness remains unconfirmed are the words $[w^q,{}_ny]$ dealt with in Theorem \ref{bb}.

Another open question related to Theorem \ref{bb} is whether the theorem remains valid if $q$ is allowed to be an arbitrary positive integer rather than a prime-power. It seems that at present we do not have sufficient tools to handle words of that kind.

The proofs of most of the known results on conciseness of words in residually finite groups use Zelmanov's theorem on nilpotency of finitely generated Lie algebras with a polynomial identity \cite{ze4}. This is also the case with both Theorem \ref{aa} and Theorem \ref{bb} of this article.

\section{Preliminaries}

Throughout the paper  we denote by $G'$ the commutator subgroup of a group $G$ and by $\langle M\rangle$ the subgroup generated by a subset $M\subseteq G$. In this section we collect some preliminary results which will be needed for the proofs of the main theorems.

A proof of the following lemma can for example be found in \cite{dms-conciseness}. 
\begin{lemma}\label{zero} Let $v$ be a word and $G$ a group such that the set of $v$-values in $G$ is finite with at most $m$ elements. Then the order of the commutator subgroup $v(G)'$ is $m$-bounded.
\end{lemma}

An important family of multilinear commutator words is formed by the derived words $\delta_k$, on $2^k$ variables, which are defined recursively by
$$\delta_0=x_1,\qquad \delta_k=[\delta_{k-1}(x_1,\ldots,x_{2^{k-1}}),\delta_{k-1}(x_{2^{k-1}+1},\ldots,x_{2^k})].$$
Of course $\delta_k(G)=G^{(k)}$, the $k$-th term of the derived series of $G$. We will need the following well-known results. 

\begin{lemma}\cite[Lemma 4.1]{S2}\label{lem:delta_k} Let $G$ be a group and  let $w$ be a multilinear commutator word on $k$ variables. Then each $\delta_k$-value is a $w$-value.
\end{lemma}

\begin{lemma}\cite[Lemma 2.2]{fernandez-morigi} \label{inverse} Let $w$ be a multilinear commutator word. Then $G_w$ is symmetric, that is, 
$x\in G_w$
implies that $x^{-1}\in G_w$. 
\end{lemma}

An element $g$ of a group $G$ is called a (left) Engel element if for any $x\in G$ there exists $n=n(g,x)\geq 1$ such that $[x,_n g]=1$. 
If $n$ can be chosen independently of $x$, then $g$ is a (left) $n$-Engel element. 
 The  next lemma is a combination of a well-known theorem of Baer  \cite[Satz~III.6.15]{hup} and a result due to Gruenberg \cite{gruen}. 

\begin{lemma}\label{fitting} Let $G$ be a group generated by (left) Engel elements.
\begin{enumerate}
\item If $G$ is finite, then it is nilpotent.
\item If $G$ is soluble, then it is locally nilpotent.
\end{enumerate}
\end{lemma}

An element $g\in G$ is a right Engel element if for each $x\in G$ there exists a positive integer $n$ such that $[g,{}_nx]=1$. If $n$ can be chosen independently of $x$, then $g$ is a right $n$-Engel element. The next observation is due to Heineken (see \cite[12.3.1]{Rob}). 

\begin{lemma}\label{heineken} Let $g$ be a right $n$-Engel element in a group $G$. Then $g^{-1}$ is a left $(n+1)$-Engel element.
\end{lemma}

A proof of the next two results can be found in \cite{dms-conciseness}. 

\begin{lemma}\label{normal closure}
Let $G=\langle g,t\rangle$ be a group such that $[g,{}_nt]=1$. Then the normal closure of the subgroup $\langle g\rangle$ in the group $G$ is generated by the set $\{g^{t^i}|i=0,\dots, n-1\}$.
\end{lemma}

\begin{lemma}\label{criterion for nilpotency}
 Let $G=U\langle t\rangle$ be a group that is a product of a normal subgroup $U$ and a cyclic 
subgroup $\langle t\rangle$. Assume that  $U$ is nilpotent of class $c$ and there exists a generating set $A$ of $U$ such that 
$[a,{}_nt]=1$ for every $a\in A$. Then $G$ is nilpotent of $(c,n)$-bounded class. 
\end{lemma}

For a subgroup $A$ of a group $G$, an element  $x \in G$ and a positive  integer $n$, we write $[A,{}_n x]$ to denote the subgroup generated by all elements $[a,{}_n x]$, with $a\in A$. The following lemma is due to Casolo.

\begin{lemma}\cite[Lemma 6]{Casolo} \label{casolo}  Let $A$ be an abelian group, and let $x$ be an automorphism of $A$ such that $[A,{}_nx]=1$ for some $n\ge1$. If $x$ has finite order $q$, then $[A^{q^{n-1}},x]=1$.
\end{lemma} 

The following proposition is a consequence of a  result from \cite{var} (see \cite[Proposition 1]{dms-conciseness}). 
\begin{proposition}\label{var22} Given positive integers $d,q,n$ and a multilinear commutator word $w$, let $G$ be a finite group in which the $w^q$-values are $n$-Engel. Suppose that a subgroup $H$ can be generated by $d$ elements which are $w^q$-values. Then $H$ is nilpotent with $(d,q,n,w)$-bounded class.
\end{proposition}

We will also require the next variation on the same theme.

\begin{proposition}\cite[Proposition 14]{BSTT}\label{non22}  Given positive integers $q,n$ and a multilinear commutator word $w$, let $G$ be a residually finite group in which the $w^q$-values are $n$-Engel. Suppose that a subgroup $H$ can be generated by finitely many Engel elements. Then $H$ is nilpotent.
\end{proposition}

The following lemma is taken from \cite{danilo}.
\begin{lemma}\cite[Lemma 4.1]{danilo}\label{danilo} Let G be a group generated by $d$ elements which are $n$-Engel. Suppose that $G$ is soluble with derived length $s$. Then $G$ is nilpotent with $(d,n,s)$-bounded class.
\end{lemma}

\begin{lemma}\label{H}
Let $n,d,q$ be positive integers. Suppose that $w=w(x_1,\ldots,x_k)$ is a multilinear commutator word and $v=[w^q,{}_ny]$.
Let $G$ be a residually finite group such that 
 $v(G)$ is abelian.
Let $g_1, \dots , g_d$ be $w^q$-values which are right $(n+1)$-Engel.
Then for every $t \in G$ the subgroup $\langle g_1, \dots , g_d, t \rangle$ is nilpotent of   $(d,n,w,q)$-bounded class.
\end{lemma}
\begin{proof}
It is sufficient to prove the result in the case where $G$ is finite. 
Let $H$ be the  normal closure of the subgroup  $\langle g_1, \dots , g_d \rangle$ in 
 $\langle g_1, \dots , g_d, t \rangle$. By Lemma \ref{normal closure}, $H$ is
  generated by at most $(n+1)d$ conjugates of the elements $g_1, \dots , g_d$, which are all  right $(n+1)$-Engel. 
Note that, by Lemma \ref{heineken} and Lemma \ref{inverse}, all $w^q$-values are 
left  $(n+1)$-Engel in  $G/v(G)$.
 As $Hv(G)/v(G)$ is generated by $(n+1)d$ elements which are $w^q$-values, by Proposition \ref{var22} it is 
 nilpotent with $(d, n,w,q)$-bounded class.
   Thus $H$ is soluble with  $(d, n,w,q)$-bounded derived length.
 As $H$ is generated by at most $(n+1)d$ element which are $(n+2)$-Engel, 
 by applying Lemma \ref{danilo}, we obtain
 that $H$ is  nilpotent with $(d,n,w,q)$-bounded class. 
Now the result follows from Lemma \ref{criterion for nilpotency}. 
\end{proof}

The next lemma can be seen as an analogue of Theorem \ref{aa} for soluble groups. The result holds for any $q\ge 1$ and $G$ here is not required to be residually finite. 

\begin{lemma}\cite[Lemma 11]{dms-conciseness}\label{soluble} Let $m,n,q,s$ be positive integers. Suppose that $w=w(x_1,\ldots,x_k)$ is a multilinear commutator word and $v=[w^q,{}_ny]$. Assume that $G$ is a soluble group of derived length $s$ such that $v$ has at most $m$ values in $G$. Then the order of $v(G)$ is $(v,m,s)$-bounded.
\end{lemma}

A proof of the next lemma can be found in \cite{dms-conciseness}.
\begin{lemma}\cite[Lemma 10]{dms-conciseness}\label{Pavel} Let $w=w(x_1,x_2,\dots,x_k)$ be a word. Let $G$ be a nilpotent group of class $c$ generated by $k$ elements $a_1,a_2,\dots,a_k$. Denote by $X$ the set of all conjugates in $G$ of elements of the form $w(a_1^i,a_2^i,\dots,a_k^i)$, where $i$ ranges over the set of all integers, and assume that $X$ is finite with at most $m$ elements. Then $|\langle X\rangle|$ is $(c,m)$-bounded.
\end{lemma}

\begin{corollary}\label{short} Let $m,n,c$ be positive integers and $K$ a nilpotent group of class at most $c$. Let $g\in K$ and denote by $Y$ the set of conjugates of elements of the form $[g,{}_nx]$, where $x\in K$. Assume that $Y$ is finite with at most $m$ elements. Then each element in $Y$ has finite order bounded by a function 
of $c$
and $m$.
\end{corollary}
\begin{proof} Choose $t\in K$. Note that 
$$[g,{}_{n}t]=[(t^{-1})^{g t},{}_{n-1} t]=[t^{-g},{}_{n-1}t]^{t},$$ and so $[g,{}_{n}t]$ is a value of the word $u=[x_1,_{n-1}x_2]$ in the subgroup $\langle t^{-g},t\rangle$ of $K$. As $u((t^{-g})^i,t^i)=[g,{}_{n}t^i]^{t^{-i}}\in Y$, the set of all conjugates in $\langle t^{-g},t\rangle$ of elements of the form $u((t^{-g})^i,t^i)$, where $i$ is any  integer, has at most $m$ elements. It follows from Lemma \ref{Pavel} that the element $[g,{}_{n}t]$ has $(c,m)$-bounded order, as required. 
\end{proof}

\section{Combinatorics of multilinear commutators}

Let $k,n$ be positive integers. Throughout this section, $w=w(x_1,\dots,x_k)$ denotes an arbitrary but fixed multilinear commutator
on $k$ variables while $v$ denotes the word $v=[w,{}_ny]$.

If $\U$ is a set of words and $G$ is a group, then $\U(G)$ will denote the subgroup generated by all values in $G$ of the words 
$u \in \U$.  

We denote by $\mathbf{I}$ the set of all $k$-tuples $(i_1,\dots,i_k)$, where all entries $i_s$ are non-negative integers. 
We will view $\mathbf{I}$ as a partially ordered set with the partial order defined by the rule that $$(i_1,\dots,i_k)\leq(j_1,\dots,j_k)$$ 
if and only if $i_1\leq j_1,\dots,i_k\leq j_k$.

Given 
 $\mathbf{i}=(i_1,\ldots,i_k) \in \mathbf{I}$, we write
\[
w_{\mathbf{i}}=w(\delta_{i_1},\ldots,\delta_{i_k})
\]
 where the variables appearing in the $\delta_{i_j}$'s are renamed in such a way that they are all different. 
Further, let 
\[
\W_{\mathbf{i^+}}=\{w_{\mathbf{j}}|\,  \mathbf{j} \in \mathbf{I};\, \mathbf{j}>\mathbf{i}\}.
\]

The following lemmas are Corollary 6 and Proposition 7 in \cite{DMScov}.
\begin{lemma}\label{abelian}
 Let $G$ be a group and $\mathbf{i}\in \mathbf{I}$. If $\W_{\mathbf{i^+}}(G)=1$,  then $w_{\mathbf{i}}(G)$ is abelian.
 \end{lemma}
  \begin{lemma}\label{pow_old}
 Let $G$ be a group  and $\mathbf{i}\in \mathbf{I}$. Suppose that 
 $a_j\in G^{(i_j)}$ for $j = 1,\dots, k$, and  let $b_s\in 
 G^{(i_s)}$. If $\W_{\mathbf{i^+}}(G)=1$, then 
$$w(a_1 ,\dots, a_{s-1}, b_s a_s , a_{s+1} ,\dots, a_k )$$
$$\quad=w(\bar a_1 ,\dots, \bar a_{s-1}, b_s ,\bar  a_{s+1} ,\dots,\bar  a_k )w(a_1 ,\dots, a_{s-1}, a_s , a_{s+1} ,\dots, a_k ),$$
where $\bar a_j$ is a conjugate of $a_j$ and moreover $\bar a_j=a_j$ if $i_j\le i_s$ .
 \end{lemma}
  \begin{corollary}\label{pow} 
 Let $G$ be a group and $\mathbf{i}\in \mathbf{I}$.
 If $\W_{\mathbf{i^+}}(G)=1$, then every power of a $w_{\mathbf{i}}$-value is  a $w$-value.
 \end{corollary}
\begin{proof}
 Let $w(a_1 ,\dots, a_k )$ be a $w_{\mathbf{i}}$-value, where each $a_j$ is a $\delta_{i_j}$-value. In particular $a_j\in G^{(i_j)}$. Let $i_s$ be the maximum among all $i_j$ with $j = 1,\dots, k$ and let $l$ be any integer. Then Lemma \ref{pow_old} and  induction on $l$ show that:
 $$w(a_1 ,\dots, a_{s-1},  a_s, a_{s+1} ,\dots, a_k )^l=w(a_1 ,\dots, a_{s-1},  a_s^l , a_{s+1} ,\dots, a_k ).$$
 This proves the result.
\end{proof}

We say that a word $u$ is an {\em extension} of $w$ if $u=w(u_1,\dots,u_k)$, for some words $u_1,\dots,u_k$. The series constructed in the following proposition
is similar to those appearing in \cite{fernandez-morigi,DMScov,fealcobershu}

\begin{proposition}\label{w-series}
Let $F$ be the free group on countably many free generators. Then for each $s\ge k$ there exist 
a $(k,s)$-bounded integer $r$ and a series 
of  length $r$
 $$F^{(s)}=\U_{0}(F)\leq \U_1(F)\le \dots\leq \U_r(F)=w(F)$$ 
such that:
\begin{enumerate}
 \item  Each $\U_j$ is a finite set of multilinear commutators which are extensions of $w$.
  \item  Every section $\U_{j+1}(F)/\U_j(F)$ is  abelian.
  \item For every 
 $j=1,\ldots,r$ and $u\in \U_j$ the inclusion
 $(F_u)^i\subseteq F_w \U_{j-1}(F)$ holds for all integers $i$, that is, $i$-th powers of $u$-values are $w$-values modulo $\U_{j-1}(F)$.
 \end{enumerate}
\end{proposition}
\begin{proof}
 We set $\U_{0}=\{\delta_{s}\}$. Note that there are only $(k,s)$-boundedly
 many
 $k$-tuples ${\bf i}$ such that $w_{\bf i}(F)\not\le \U_{0}(F)$. 
 Choose ${\bf i}$ such that $w_{\bf i}(F)\not\le \U_{0}(F)$ but
 $\W_{\mathbf{i^+}}(F)\le \U_{0}(F)$ and set  $\U_{1}=\{w_{\bf i}\}\cup \U_0$. More generally, once we have defined $\U_{j-1}$,
 we choose 
 ${\bf i}$ such that $w_{\bf i}(F)\not\le \U_{j-1}(F)$ but
 $\W_{\mathbf{i^+}}(F)\le \U_{j-1}(F)$ and set  $\U_{j}=\{w_{\bf i}\}\cup \U_{j-1}$. After a $(k,s)$-bounded number of steps $r$ we get 
  ${\bf i}=(0,\dots,0)$. So $\U_r=\{w\}\cup\U_{r-1}$.  As all the words in $\U_{r-1}$ are extensions of $w$ we obtain that $ \U_r(F)=w(F)$.
 
 As  $\W_{\mathbf{i^+}}(F)\le \U_{j-1}(F)$, it follows that 
$$\U_j(F)/\U_{j-1}(F)=w_{\bf i}(F)\U_{j-1}(F)/\U_{j-1}(F)$$ 
 is a homomorphic image of $w_{\bf i}(F)/\W_{\mathbf{i^+}}(F)$.
 So the quotient group $\U_j(F)/\U_{j-1}(F)$   is abelian by Lemma \ref{abelian}. 
Moreover, by Corollary \ref{pow},
 powers of $w_{\bf i}$-values are again  $w$-values  modulo $\W_{\mathbf{i^+}}(F)$ so the same holds  modulo $\U_{j-1}(F)$.
 \end{proof}

\begin{proposition}\label{ri}
 Let $\U_0,\dots, \U_r$ be sets of words as in Proposition \ref{w-series}. Then for every positive integer $i$  there exists an integer $r_i$
 such that for every group $G$, for every $j=1,\dots.r$, for every $\U_j$-value $g$ in $G$ and every $t\in G$ we have
 \begin{equation*}
[g,{}_n t]^i= \h \tau,
\end{equation*}
where $\h$ is a $v$-value and $\tau$ is the product of $r_i$ $\U_{j-1}$-values.
 \end{proposition}

\begin{proof} 
  Let $F$  be the free group of countable rank with free generators $\{y,x_\eta| \eta=1,2,\dots\}$
 and let
 $$F^{(s)}=\U_{0}(F)\leq \U_1(F)\le \dots\leq \U_r(F)=w(F),$$ 
 be a series as   in Proposition \ref{w-series}.
Consider an arbitrary $j$ and let $u\in \U_j$. 
Recall that  $u$-values are $w$-values
 and their powers are  $w$-values modulo $\U_{j-1}(F)$.
 So for every integer $i$ we have 
 $$u^i\U_{j-1}(F)=\tilde u_i \U_{j-1}(F),$$
 for some $w$-value $\tilde u_i$ in $F$ (which depends on $i$). As $\U_j(F)/\U_{j-1}(F)$ is abelian, it follows that
 $$[u,{}_n y]^i\U_{j-1}(F)=[u^i,{}_n y]\U_{j-1}(F)=[\tilde u_i,{}_n y]\U_{j-1}(F).$$
 So  
\begin{equation*}
[u,{}_n y]^i=[\tilde u_i,{}_n y]\tau_i,
\end{equation*}
where $[\tilde u_i,{}_n y]$ is a $v$-value in $F$ and $\tau_i$ is the product of  a finite number 
 $r(i, u)$ of $\U_{j-1}$-values. 
Let $r_i$ 
be the maximum of all $r({i,u})$'s,  for all words $u\in \U_j$ and 
 all  $j=0,\dots,r$. Note that $r_i$ depends only on  $w$ and $i$.

Now let $G$ be any group and let $g$ be a $u$-value for some $u\in \U_j$, say $u=u(x_1,\dots,x_l)$, so that $g=u(g_1,\dots,g_l)$ where each $g_\eta$ 
belongs to $G$. Consider the epimorphism $\varphi: F\to \langle g_1,\dots,g_l, t\rangle$ defined by  $\varphi(y)=t$,
$\varphi(x_\eta)=g_\eta$ if $\eta\in\{1,\dots,l\}$,
$\varphi(x_\eta)=1$ otherwise. Then $[g,{}_n t]^i=\varphi([u,{}_n y]^i)$ has the desired properties, where $\h=\varphi([\tilde u_i,{}_n y])$ and
$\tau=\varphi(\tau_i)$.
\end{proof}

\section{Proof of Theorem \ref{aa}}

In this section we will prove the first of our main results. 

\begin{proof}[Proof of Theorem \ref{aa}]
Let $v=[w,{}_n\,y]$, where  $w$ is a   multilinear commutator word
 in $k$ variables 
 and  $n \ge 1$.   We need to show that if $G$ is a residually finite group with at most $m$ values of the word $v$, then the order of $v(G)$ is $(v,m)$-bounded. Evidently, it suffices to establish this result for finite quotients of $G$. Therefore, without loss of generality, we assume that $G$ is finite. 

In view of Lemma \ref{zero} we may pass to the quotient $G/v(G)'$ and assume that $v(G)$ is abelian. The image of $w(G)$ in $G/v(G)$ is generated by right Engel elements. Hence $w(G)/v(G)$ is nilpotent. 
Since $w$ is a multilinear commutator on $k$ variables, Lemma \ref{lem:delta_k} tells us that every $\delta_k$-value is a $w$-value. It follows that the  group $G$ is soluble. Taking into account that $C_G(v(G))$ has $m$-bounded   
 index in $G$, we deduce that
$G/C_G(v(G))$  has $m$-bounded derived length. 
Let $l$ be the smallest integer greater than $k$ such that 
 $G^{(l)}$ centralizes $v(G)$. Note that $l$ is $(m,k)$-bounded and
every $\delta_l$-value is also a $w$-value. 

Let $x,y\in G$ where $y$ is a $\delta_l$-value. Using the formula $[x,y,y]=[y^{-xy},y]$ and taking into account that $y^{-1}$ is 
a $w$-value, 
 we deduce that the commutator $[x,{}_{n+1}\,y]$ is a $v$-value and therefore, since 
$[v(G),y]=1$, we have $[x,{}_{n+2}\,y]=1$. 
 So  each $\delta_l$-value is $(n+2)$-Engel in $G$. 

Note that any $\delta_{2l}$-value can naturally be written as a $\delta_l$-value whose entries are also $\delta_l$-values. Let  $g=\delta_l(g_{1},\dots,g_{2^{l}})$ be a $\delta_{2l}$-value, 
where each $g_{i}$ is a $\delta_l$-value. Further, choose an arbitrary element $t\in G$ and let $H$ be the minimal $t$-invariant subgroup 
of $G$ containing all these elements $g_{i}$. Since the image of $t$ in $G/v(G)$ acts on each $g_{i}$ as an Engel element, 
Lemma \ref{normal closure} tells us that the image of $H$ in $G/v(G)$ is generated by a $(v,m)$-bounded number of
 $\delta_l$-values,
which are $(n+2)$-Engel. By  Proposition \ref{var22}, 
$ H{v(G)}/{v(G)}$  is nilpotent of $(v,m)$-bounded class.
Since ${v(G)}$ is abelian, we conclude that $ H\langle  t\rangle$ is soluble  of $(v,m)$-bounded derived length. 
Lemma \ref{soluble} tells us that $v(H\langle t\rangle)$ is finite with order bounded by an integer $R=R(v,m)$ depending only on $v$ and $m$.   
Since $v(G)$ is abelian of rank at most $m$, the order of the subgroup $M_{0}$ of $v(G)$ generated by all elements of order at most $R$ is $(v,m)$-bounded.  In particular, as $[g,{}_nt]\in v(H\langle t\rangle)$,  it follows that $[g,{}_nt]\in M_0$. Thus we have found a normal subgroup $M_{0}$ of $(v,m)$-bounded order such that for every $\delta_{2l}$-value $g$, and every  $t\in G$ we have that  $[g,{}_nt]=1$  in the quotient group $G/M_0$.

Consider now the sets of words 
 $\U_0,\dots, \U_r$  defined  in Proposition \ref{w-series}   and  the integers $r_i$ 
 defined  in Proposition \ref{ri} with $s=2l$.
 Then for every integer $i$, for every $j=1,\dots ,r$, for every $\U_j$-value $g\in G$ and every $t\in G$ we have
 \begin{equation}\label{tauGG}
[g,{}_n t]^i= \h \tau,
\end{equation}
where $\h$ is a $v$-value and $\tau$ is the product of $r_i$ $\U_{j-1}$-values.

We will prove by induction on $j$ that 
\begin{itemize}
  \item[(*)]
$G$ has a normal subgroup $M_j\leq v(G)$ of $(v,m,j)$-bounded order such that $[g,{}_{n+1}t]\in M_j$ for every $\U_j$-value $g$ and every $t\in G$.
 \end{itemize}

In the case   $j=0$ we have 
  $\U_0=\{\delta_{2l}\}$, and the result has already been proved. 

Assume that $j>0$. Without loss of generality we can assume that $M_{j-1}$ is trivial and so every 
$\U_{j-1}$-value is right $(n+1)$-Engel in $G$. 
 
Let $g$ be an arbitrary $\U_j$-value in $G$ and let $t\in G$. 
  As 
in $G$ there are at most $m$ $v$-values, equality (\ref{tauGG}) implies that  for some $0 \le i_1 < i_2\leq m $ 
we have 
$$[g,{}_nt]^{i_1}=\h \tau_{1},$$ 
$$[g,{}_nt]^{i_2}=\h \tau_{2},$$ 
where $\tau_1$ and $\tau_2$ are products of at most $r_1$ and $r_2$  $\U_{j-1}$-values, respectively. 
So  $[g,{}_nt]^{i_2-i_1}$ is the product of 
  $f$ (possibly trivial) $\U_{j-1}$-values
$z_1,\dots,z_f$ 
where 
$$f=2 \cdot \max \{ r_i \mid 0 \le i \le m \} .$$
 
By induction, $[z_i,{}_{n+1}t]=1$ for each $i$. Lemma \ref{H} guarantees that the subgroup $Z=\langle  z_1,\ldots,z_f, t \rangle$ is nilpotent of $(v,m)$-bounded class. 
 So there exists a $(v,m)$-bounded  integer $s$ such that $[Z,{}_st]=1$. 
Since the index of $C_G(v(G))$ in $G$ is finite, conjugation by $t$ is an automorphism of $v(G)\cap Z$ of bounded order, say $e$. 
In view of  Lemma \ref{casolo} we obtain that $[(v(G)\cap Z)^{e^{s-1}},t]=1$. As $v(G)\cap Z$ is abelian, it follows that 
$[v(G)\cap Z,t]$ has finite exponent at most $e^{s-1}$, which is  $(v,m)$-bounded.
Let $N$ be the subgroup of $v(G)$ generated by all elements of order at most $e^{s-1}$. Note that $N$ has $(v,m)$-bounded order. 

If $N=1$, then $[a,t]=1$ for every $a\in v(G)\cap Z$ and $t\in G$, that is, $v(G)\cap Z$ is contained in the center
of the group $G$. 
Moreover, as $[g,{}_nt]^{{i_2-i_1}}\in Z\cap v(G)$ and 
 $v(G)$ is  abelian, 
 we deduce that 
$$1=[[g,{}_nt]^{ {i_2-i_1}},t]= [[g,{}_nt],t]^{{i_2-i_1}  }=[g,{}_{n+1}t]^{ {i_2-i_1} }.$$

 Let $M_j$ be the subgroup generated by all elements of $v(G)$ whose ${m!}$-th power belongs to $N$.  The above equality shows that $[g,{}_{n+1}t]\in M_j$ for an arbitrary  $\U_j$-value $g$ and an arbitrary $t\in G$.
 Obviously, the order of $M_j$ is bounded. 
This concludes the proof of (*).

Therefore there exists a normal subgroup $M_r$ of $G$ of $(v,m)$-bounded order such that in the quotient group $G/M_r$ the equality $[g,{}_{n+1}t]=1$ holds for every $w$-value $g$ and for every element $t$.

Passing to the quotient $G/M_r$, we can assume that $M_r=1$. By Lemma \ref{H} the subgroup $\langle g,t\rangle$ is nilpotent of $(v,m)$-bounded class  for every $g\in G_w$ and every $t\in G$. In view of Corollary \ref{short} we obtain that for every $g \in G_w$ and every $t \in G$ the element $[g,{}_{n}t]$ has $(v,m)$-bounded order. 

Thus, $v(G)$ is an abelian group of rank at most $m$ generated by elements of $(v,m)$-bounded order. Hence, the order of $v(G)$ is $(v,m)$-bounded. The proof is complete.
\end{proof}

\section{Theorem \ref{bb}}

In the present section Theorem \ref{bb} will be proved. We need the following proposition, which is the main result in \cite{P05}.
\begin{proposition}\label{pp-order}
Let $G$ be a residually finite group satisfying some identity. 
Suppose $G$ is generated by a normal commutator-closed set of $p$-elements. Then $G$ is locally finite.
\end{proposition}

Recall that a group is locally graded if every non-trivial finitely generated subgroup has a proper subgroup of finite index. The class
of locally graded groups is fairly large and in particular it contains all residually finite groups. Proposition \ref{pp-order} can be easily extended to the case of locally graded groups. 

\begin{lemma}\label{locgrad}
Let $G$ be a locally graded group satisfying some identity. Suppose  $G$ is generated by a normal commutator closed set 
 of $p$-elements for some prime $p$. Then $G$ is locally finite.
\end{lemma}
\begin{proof} Let $X$ be a normal commutator closed set of  $p$-elements which generate $G$.
 It is sufficient to prove that any subgroup generated by finitely many elements from $X$ is finite. 
Let  $K=\langle x_1,\dots,x_s\rangle$, where $x_1,\dots,x_s\in X$. Let $R$ be the finite residual of $K$.
 It follows from Proposition \ref{pp-order} 
that $K/R$ is 
 finite. In particular $R$  is 
finitely generated. 
 Now it is enough to prove that $R=1$. Assume by contradiction that $R$ is nontrivial. 
As $G$ is locally graded, $R$ has a proper subgroup $N$ of finite index. 
Clearly,  $N$ has also finite index in $K$, and as $R$
is the finite residual of $K$ it follows that $R\le N$, a contradiction.
\end{proof}


Important properties of the verbal subgroup corresponding to a multilinear commutator word in a soluble-by-finite group are described
in the following proposition.

\begin{proposition}\cite[Proposition 2.6]{AS1} \label{series}
Let $w$ be a multilinear commutator word, and let $G$ be  a  group having a normal soluble subgroup of finite index $e$ and derived length $s$. Then $G$ has a series of  subgroups 
$$1 = T_1 \le T_2\le \dots\le T_l = w(G)$$
such that: 
\begin{enumerate}
\item All subgroups $T_i$ are normal in $G$.
\item The length $l$ of the series is $(w, s,e)$-bounded.
\item Every section $T_{i+1}/T_i$ is abelian and can be generated by $w$-values in $G/T_{i}$ all of whose powers are also $w$-values, except possibly
one section whose order is finite and $(s,e)$-bounded. 
\end{enumerate}
\end{proposition}

The next lemma is a version of Theorem \ref{bb} in the case when the group $G$ is soluble-by-finite. 
 Note that $G$ here is not required to be residually finite and $q$ is an arbitrary integer. 

\begin{proposition}\label{soluble-by-fin} Let $m,n,q,s,e$ be positive integers. Suppose that $w=w(x_1,\ldots,x_k)$ is a multilinear commutator word and 
$v=[w^q,{}_ny]$. Assume that $G$ has  a normal soluble subgroup of finite index $e$ and derived length $s$ and suppose  that $v$ has at most $m$ values in $G$. 
Then the order of $v(G)$ is $(v,m,s,e)$-bounded.
\end{proposition}
\begin{proof} In view of Lemma \ref{zero} we may assume that $v(G)$ is abelian. 
Consider a series  as in Proposition \ref{series}.
We will use induction on the length of this series, the case $w(G)=1$ being trivial. Let $L$ be the last nontrivial term of the series. By induction we assume that the image of $v(G)$ in $G/L$ has finite $(v,m,s,e)$-bounded order. Set $K=v(G)\cap L$. It follows that the index of $K$ in $v(G)$ is $(v,m,s,e)$-bounded. 

Without loss of generality, we may assume that the subgroup $L$ is abelian and can be generated by $w$-values in $G$ all of whose powers are also $w$-values. 
Let $g\in L$  be one of those $w$-values and let $t\in G$. Then for every positive integer $i$ the element $[g^{iq},{}_nt]$ is a $v$-value. As $L$ is 
 abelian, it follows that $[g^{iq},{}_nt]=[g^q,{}_nt]^i$.
Since $v$ has at most $m$ values in $G$, 
 there are two different integers $i_1,i_2$ with $0\le i_1<i_2\le m$ such that 
 $[g^q,{}_nt]^{i_1}=[g^q,{}_nt]^{i_2}$. It follows that $[g^q,{}_nt]$ has order at most $m$, and consequently $[g,{}_nt]$ has order at most $mq$. 
 Let $T_1$ be the subgroup of $v(G)$ generated by all elements of order at most $mq$. As $v(G)$ is abelian with at most $m$ generators, $T_1$ 
has $(m,q)$-bounded order. Thus, we can pass to the quotient $G/T_1$ and assume that $[g,{}_nt]=1$ for every generator $g$ of $L$ chosen as above and for every $t\in G$. Since $L$ is abelian, it follows that $[L,{}_nt]=1$ for every $t\in G$. In particular, $[K,{}_n t]=1$  for every $t\in G$.

Since the index of $C_G(v(G))$ in $G$ is $m$-bounded, the conjugation by an arbitrary element $t\in G$ induces an automorphism of $v(G)$ of $m$-bounded order, say $r$. Lemma \ref{casolo} tells us that $[K^{r^{n-1}},t]=1$. As $K$ is abelian, it follows that $[K,t]$ has exponent dividing $r^{n-1}$, which is $(m,n)$-bounded. Let $T_2$ be the subgroup of $v(G)$ generated by all elements of order at most $r^{n-1}$. We can pass to the quotient $G/T_2$ and without loss of 
generality assume that $[K,t]=1$ for every $t\in G$. Therefore $K$ is contained in the center of the group $G$. Further, 
note that $K\langle t^r\rangle$ is a central subgroup of $v(G)\langle t\rangle$ and has $(v, m,s,e)$-bounded index in $v(G)\langle t\rangle$. So
by Schur's Theorem \cite[10.1.4]{Rob} the derived subgroup of $v(G)\langle t\rangle$ has $(v,m,s,e)$-bounded order, and the bound  does not depend on the choice of $t$. 
Arguing as before and factoring out an appropriate  subgroup of $v(G)$ of bounded order, we may assume that $[v(G),t]=1$ for every $t\in G$, 
that is, $v(G)$ is contained in the center of $G$.

In particular, $[g^q,{}_{n+1}t]=1$ for every $g\in G_w$ and every $t\in G$. So every $w^q$-value is right $(n+1)$-Engel in $G$. 
Thus, by Lemma \ref{heineken} combined with Lemma \ref{inverse}, every $w^q$-value is left $(n+2)$-Engel.  

 It follows from Lemma \ref{normal closure} that the normal closure $U$ of the subgroup $\langle g^q\rangle$
in the group $\langle g^q,t\rangle$ is generated by the set $A=\{(g^q)^{t^i}|i=0,\dots, n\}$ whose elements are left $(n+2)$-Engel. The subgroup $U$ has a  
normal soluble subgroup $N$ of index at most $e$ and derived length at most $s$.
It follows from Lemma \ref{fitting} that $U/N$ is nilpotent, and as it has bounded order, in particular it is soluble of
$e$-bounded derived length. 
 Lemma \ref{danilo} now tells us that $U$ is nilpotent with $(n,s,e)$-bounded class. 
 As $[a,{}_{n+1}t]=1$ for every $a\in A$, 
 Lemma \ref{criterion for nilpotency} shows that  $U\langle t\rangle=\langle g^q,t\rangle$ is nilpotent of $(v,m,s,e)$-bounded class.

By Corollary \ref{short} the element $[g^q,{}_{n}t]$ has 
$(v,m,s,e)$-bounded order. 
Thus, we have shown that $v(G)$ is an abelian group of rank at most $m$ generated by elements of $(v,m,s,e)$-bounded order. 
Hence, the order of $v(G)$ is  $(v,m,s,e)$-bounded. The proof is complete.
\end{proof}

\begin{proof}[Proof of Theorem \ref{bb}]
Let $G$ be a residually finite group with finitely many $v$-values,  
 where $v=[w^q,{}_ny]$, $w$ is  multilinear commutator word and $q$ is a prime power. 
In view of Lemma \ref{zero} we may pass to the quotient $G/v(G)'$ and assume that $v(G)$ is abelian. Since $v(G)$ is finitely generated, it is clear that elements of finite order 
in $v(G)$ form a finite normal subgroup. We pass to the quotient over this subgroup and without loss of generality assume that $v(G)$ is torsion-free.

Since $G$ is residually finite, we can choose a normal subgroup $H$ of finite index such that $H$ intersects $G_v$ trivially. Observe that every $w^q$-value in $H$ is right $n$-Engel. It follows from Lemmas \ref{heineken} and \ref{inverse} that every $w^q$-value in $H$ is left $(n+1)$-Engel.
Let $K=w^q(H)$. It follows from Proposition \ref{non22} that $K$ is locally nilpotent. Moreover, Lemma \ref{criterion for nilpotency} implies that for each $t\in G$ the subgroup $\langle K,t\rangle$ is locally nilpotent. Note that  $G/K$ need not be residually finite but by the result in \cite{LMS} the quotient $G/K$ is locally graded. We deduce from Lemma \ref{locgrad} that $w(H)/K$ is locally  finite.

Note that in the group $\bar G=G/w(H)$ the subgroup $\bar H$ is a soluble normal subgroup of finite index. Proposition \ref{soluble-by-fin} guarantees that $v(\bar G)$ is finite. Thus, $v(G)\cap w(H)$ has finite index in $v(G)$. 

Using the fact that  $v(G)$ is finitely generated and the 
 local finiteness of $w(H)/K$ we 
deduce that $v(G)\cap K$ has finite index in $v(G)\cap w(H)$. Thus $v(G)\cap K$ has finite index in $v(G)$. 
In particular, $v(G)\cap K$  is finitely generated.

Choose $t\in G$. Since the subgroup $\langle K,t\rangle$ is locally nilpotent and $v(G)\cap K$ is finitely generated,
there exists an integer $s$ such that $[v(G)\cap K,{}_st]=1$. 

Since the index of $C_G(v(G))$ in $G$ is finite, the
conjugation by  $t$ induces an automorphism of $v(G)$ of finite 
 order, say $r$. Lemma \ref{casolo} tells us that $[(v(G)\cap K)^{r^{s-1}},t]=1$. As 
$v(G)\cap K$ is abelian, it follows that $[v(G)\cap K,t]$ has finite exponent dividing $r^{s-1}$.
 Since $v(G)$ is torsion-free, $[v(G)\cap K,t]=1$.

We see that $(v(G)\cap K)\langle t^r\rangle$ is a central subgroup of finite index  in $v(G)\langle t\rangle$.  
 So by Schur's Theorem the commutator subgroup of $v(G)\langle t\rangle$ is finite. Since the commutator subgroup is contained in $v(G)$, 
 which is torsion-free, it follows that $[v(G), t]=1$. Since $t$ is arbitrary, $v(G)$ is contained in the center of $G$. 
In particular, if $g\in G_w$ and $t\in G$, we have $[g^q,{}_{n+1}t]=1$. 
 
 Hence, by Lemma \ref{H} $\langle g^q, t \rangle$ is nilpotent and so 
 by Corollary \ref{short} we obtain that  $[g^q,{}_{n}t]$ has finite order. 

Since $[g^q,{}_{n}t]$ is  an arbitrary $v$-value and $v(G)$ is torsion-free, we conclude that $v(G)=1$. The theorem is established.
\end{proof}

\end{document}